\documentclass[11pt,a4paper]{amsart}

\usepackage{amsfonts}

\usepackage[T1]{fontenc}

\usepackage{enumerate}

\usepackage[abbrev]{amsrefs}

\usepackage[utf8]{inputenc}

\usepackage{amssymb}

\usepackage{palatino}

\usepackage{color}

\renewcommand{\div}{\operatorname{div}}
\newcommand{\abs}[1]{\mathopen\lvert#1\mathclose\rvert}
\newcommand{\biggabs}[1]{\biggl\lvert#1\biggr\rvert}
\newcommand{\norm}[1]{\mathopen\lVert#1\mathclose\rVert}
\newcommand{\N}{{\mathbb N}}
\newcommand{\R}{{\mathbb R}}
\newcommand{\cH}{\mathcal{H}}
\newcommand{\cM}{\mathcal{M}}
\DeclareMathOperator{\supp}{supp}
\newcommand{\loc}{_\mathrm{loc}}
\newcommand{\dif}{\,\mathrm{d}}

\theoremstyle{plain}
\newtheorem{proposition}{Proposition}[section]
\newtheorem{lemma}[proposition]{Lemma}
\newtheorem{theorem}[proposition]{Theorem}
\newtheorem{corollary}[proposition]{Corollary}

\theoremstyle{remark}
\newtheorem{example}{Example}[section]
\newtheorem*{Claim}{Claim}

\numberwithin{equation}{section}

\title[Singularities of the divergence of continuous vector fields]{Singularities of the divergence of continuous vector fields and uniform Hausdorff estimates}

\author{Augusto C. Ponce}
\address{
Augusto C. Ponce\hfill\break\indent
 Universit{\'e} catholique de Louvain\hfill\break\indent
 Institut de Recherche en Math{\'e}matique et Physique\hfill\break\indent
 Chemin du cyclotron 2, bte L7.01.02\hfill\break\indent
1348 Louvain-la-Neuve\hfill\break\indent
Belgium}
\email{Augusto.Ponce@uclouvain.be}



\begin{document}

\begin{abstract}
We prove that every closed set which is not \(\sigma\)-finite with respect to the Hausdorff measure \(\cH^{N-1}\) carries singularities of continuous vector fields in \(\R^N\) for the divergence operator.
We also show that finite measures which do not charge sets of \(\sigma\)-finite Hausdorff measure \(\cH^{N-1}\) can be written as an \(L^1\) perturbation of the divergence of a continuous vector field.
The main tool is a property of approximation of measures in terms of the Hausdorff content.
\end{abstract}

\maketitle

\section{Introduction and main results}

The original motivation of this paper is related to the following problem: to find a simple characterization of all closed sets \(S\) such that if \(V : \R^N \to \R^N\) is a continuous vector field such that
\[
\div{V} = 0 \quad \text{in \(\R^N \setminus S\)},
\]
then
\[
\div{V} = 0 \quad \text{in \(\R^N\)}.
\]
Such sets \(S\) cannot carry singularities of continuous vector fields for the divergence operator; we say in this case that \(S\) is \(C^0\) removable. 
We show that the answer to this question is given by the following:

\begin{theorem}
\label{theoremRemovableSingularity}
Let \(S \subset \R^N\) be a closed set. Then, \(S\) is \(C^0\) removable for the divergence operator if and only if
\(S\) is \(\sigma\)-finite for the Hausdorff measure \(\cH^{N-1}\).
\end{theorem}

The reverse implication ``\(\Leftarrow\)'' has been established by de~Valeriola and  Moonens~\cite{deValeriola_Moonens}*{Theorem~12}. In Section~\ref{Section2}, we provide the direct implication ``\(\Rightarrow\)'' by showing that if \(S\) is not \(\sigma\)-finite for the Hausdorff measure \(\cH^{N-1}\), then there exists a (Borel) positive measure \(\mu\) supported on \(S\) such that the equation
\[
\div{V} = \mu \quad \text{in \(\R^N\)}
\]
has a continuous solution. 

This result completes the picture concerning the removability of singularities for \(C^0\) and \(L^\infty\) vector fields. Indeed, Theorem~\ref{theoremRemovableSingularity} has the following counterpart concerning \emph{bounded} (not necessarily continuous) vector fields:

\begin{theorem}
\label{theoremRemovableSingularityBounded}
Let \(S \subset \R^N\) be a closed set. Then, \(S\) is \(L^\infty\) removable for the divergence operator if and only if \(S\) has zero Hausdorff measure \(\cH^{N-1}\).
\end{theorem}

Theorem~\ref{theoremRemovableSingularityBounded} has been proved independently by Moonens~\cite{Moonens}*{Theorem~4.7} and by Phuc and Torres~\cite{Phuc_Torres}*{Theorem~5.1}.
We also refer the reader to \cite{DePauw}*{Theorem~6.3} in the case where \(S\) is purely unrectifiable. 

\medskip
The implication ``\(\Rightarrow\)'' in Theorem~\ref{theoremRemovableSingularity} which concerns us relies on the study of existence of continuous vector fields \(V : \R^N \to \R^N\) such that
\begin{equation}\label{eqDiv}
\div{V} = \mu \quad \text{in \(\R^N\)},
\end{equation}
where \(\mu\) is a function or a measure.
The case where \(\mu\) belongs to \(L^N(\R^N)\) has been investigated by Bourgain and Brezis~\citelist{\cite{Bourgain_Brezis}*{Proposition~1 and Remark~1} \cite{DePauw_Pfeffer}*{Proposition~2.9}}. They have given an affirmative answer using the Closed range theorem. 
This solution cannot be obtained from the equation
\[
\div{\nabla u} = \Delta u = \mu\quad \text{in \(\R^N\)}
\]
in view of the lack of embedding of the Sobolev space \(W^{1, N}\) into the space of continuous functions \(C^0\).

A characterization of finite measures --- and more generally distributions --- in \(\R^N\) for which equation \eqref{eqDiv} has a continuous solution has been obtained by De~Pauw and Pfeffer~\cite{DePauw_Pfeffer}*{Theorem~4.8}. Stated in terms of \emph{strong charges}, they have proved that equation \eqref{eqDiv} has a \(C^0\) solution if and only if for every \(\epsilon > 0\) and for every compact set \(K \subset \R^N\), there exists \(C > 0\) such that for every \(\varphi \in C_c^\infty(\R^N)\) supported in \(K\),
\[
\biggabs{\int\limits_{\R^N} \varphi \dif\mu} \le C \norm{\varphi}_{L^1(\R^N)}
 + \epsilon \norm{D\varphi}_{L^1(\R^N)}.
\]

Phuc and Torres~\cite{Phuc_Torres}*{Theorem~4.5} have shown that for nonnegative measures the condition of being a strong charge is equivalent to asking that for every compact set \(K \subset \R^N\),
\begin{equation}\label{eqPhucTorresContinuous}
\lim_{\delta \to 0}{{\sup_{\substack{x \in K\\ r \le \delta}}{\frac{\mu(B(x; r))}{r^{N-1}}}}} = 0.
\end{equation}

We may get some insight about the meaning of assumption \eqref{eqPhucTorresContinuous} using the Besicovitch covering theorem~\cite{Mattila}*{Theorem~2.7}: if \(\mu\) satisfies \eqref{eqPhucTorresContinuous} for every compact set \(K\), then for every Borel set \(A \subset \R^N\),
\begin{equation}\label{eqConditionSigmaFinite}
\cH^{N-1}(A) < +\infty \quad \text{implies} \quad \mu(A) = 0.
\end{equation}
By countable subadditivity of the measure \(\mu\) this amounts to asking that \(\mu(A) = 0\) for every Borel set \(A \subset \R^N\) which is not \(\sigma\)-finite with respect to the Hausdorff measure \(\cH^{N-1}\).	 

One does not need the full power of \eqref{eqPhucTorresContinuous} to deduce property~\eqref{eqConditionSigmaFinite} since the pointwise convergence with respect to \(x\) already suffices to obtain the same conclusion. 
We would like to understand in what sense a measure satisfying \eqref{eqConditionSigmaFinite} misses the uniform limit in \eqref{eqPhucTorresContinuous}. 

For this purpose, note that any finite measure (positive or not) in \(\R^N\) can be approximated by smooth functions, for instance via convolution, but such approximation is rather weak and holds in the sense of distributions. The convergence cannot be strong with respect to the total mass norm
\[
\norm{\mu}_{\cM(\R^N)} = \abs{\mu}(\R^N),
\] 
unless \(\mu\) is an \(L^1\) function by completeness of \(L^1(\R^N)\).

The main tool in this paper asserts that we can pass from condition \eqref{eqConditionSigmaFinite} to \eqref{eqPhucTorresContinuous} by strong convergence of measures:

\begin{proposition}
\label{theoremStrongConvergenceMeasures}
For every finite nonnegative measure \(\mu\) in \(\R^N\) satisfying \eqref{eqConditionSigmaFinite}, there exists a nondecreasing sequence of finite nonnegative measures \((\mu_n)_{n \in \N}\) in \(\R^N\) such that 
\[
\lim_{n \to \infty}{\norm{\mu_n - \mu}_{\cM(\R^N)}} = 0,
\]
and for every \(n \in \N\),
\[
\lim_{\delta \to 0}{{\sup_{\substack{x \in \R^N\\ r \le \delta}}{\frac{\mu_n(B(x; r))}{r^{N-1}}}}} = 0.
\]
\end{proposition}

Each measure \(\mu_n\) is actually obtained from \(\mu\) by restriction: although \(\mu\) need not satisfy \eqref{eqPhucTorresContinuous}, it suffices to remove a small part of \(\mu\) and the remaining of the measure verifies that property.

A similar approximation result still holds concerning any positive dimension \(s\) and the main tool is a uniform comparison principle between the Hausdorff measure \(\cH^s\) and the outer measures \(\cH^s_\delta\); see Proposition~\ref{propositionStrongEstimateHausdorffMeasure} and Proposition~\ref{propositionEpsilonApproximation}. 

\medskip

Returning to equation \eqref{eqDiv}, we have seen that if \(\mu\) satisfies condition \eqref{eqConditionSigmaFinite}, then we can extract a small part of \(\mu\) and the remaining part yields a continuous solution of the equation. 
As a consequence, we prove in Section~\ref{Section4} that there exist continuous solutions except for an \(L^1\) perturbation of \(\mu\):

\begin{theorem}
\label{theoremL1Perturbation}
For every finite measure \(\mu\) in \(\R^N\) satisfying \eqref{eqConditionSigmaFinite} and for every \(\epsilon > 0\), there exist \(V \in C^0(\R^N; \R^N)\) and \(f \in L^1(\R^N)\) such that
\[
\mu = \div{V} + f \quad \text{in \(\R^N\)}
\]
and \(\norm{f}_{L^1(\R^N)} \le \epsilon\).
\end{theorem}

The vector field \(V\) may be chosen to vanish at infinity.
This requires a variant of the result of De~Pauw and Pfeffer due to De~Pauw and Torres~\cite{DePauw_Torres}*{Theorem~6.1} for strong charges in \(\R^N\) vanishing at infinity; see remark following the proof of Theorem~\ref{theoremL1Perturbation} in Section~\ref{Section4} below.

Theorem~\ref{theoremL1Perturbation} has a counterpart for the Laplace operator, in which case the \(W^{1, 2}\) capacity plays the role of the Hausdorff measure \(\cH^{N-1}\); see \cite{Brezis_Marcus_Ponce}*{Theorem~4.3}. 
This type of characterization is reminiscent of a decomposition of Boccardo, Gallouët and Orsina~\cite{BocGalOrs:96}*{Theorem~2.1} for finite measures which are diffuse with respect to the \(W^{1, p}\) capacity for some \(1 < p < +\infty\).

\medskip

As a result of Theorem~\ref{theoremL1Perturbation}, every finite measure satisfying \eqref{eqConditionSigmaFinite} is a \emph{charge}
in the sense that for every \(\epsilon > 0\) and for every compact set \(K \subset \R^N\), there exists \(C > 0\) such that for every \(\varphi \in C_c^\infty(\R^N)\) supported in \(K\),
\[
\biggabs{\int\limits_{\R^N} \varphi \dif\mu} \le C \norm{\varphi}_{L^1(\R^N)} + \epsilon \big(\norm{D\varphi}_{L^1(\R^N)} + \norm{\varphi}_{L^\infty(\R^N)} \big).
\]
The notions of charges and strong charges have connections with the Divergence theorem and have been investigated by De~Pauw, Pfeffer and collaborators; see \cites{Pfeffer} and the references therein.

De~Pauw and Pfeffer~\cite{DePauw_Pfeffer}*{Theorem~6.2} have shown that charges coincide with strong charges modulo a locally integrable function.
Combining this result with Theorem~\ref{theoremL1Perturbation} above, we characterize in Section~\ref{Section5} finite measures in \(\R^N\) which are charges:

\begin{corollary}
\label{corollaryCharges}
For every finite measure \(\mu\) in \(\R^N\), \(\mu\) is a charge if and only if property \eqref{eqConditionSigmaFinite} holds.
\end{corollary}

By the Hahn decomposition theorem, property \eqref{eqConditionSigmaFinite} is equivalent to asking that
for every Borel set \(A \subset \R^N\),
\begin{equation*}
\cH^{N-1}(A) < +\infty \quad \text{implies} \quad \abs{\mu}(A) = 0.
\end{equation*}
We deduce from the corollary above that if a finite measure \(\mu\) is a charge, then the positive and the negative parts of \(\mu\) are also charges.
The counterpart of this property for strong charges is false in dimension \(N \ge 2\); see Example~\ref{exampleStrongCharge} below.


\section{Proof of Theorem~\ref{theoremRemovableSingularity}}
\label{Section2}

The proof of the direct implication of Theorem~\ref{theoremRemovableSingularity} relies on three main tools: the characterization of strong charges of Phuc and Torres given by condition~\eqref{eqPhucTorresContinuous}, Frostman's lemma and a property of sets which are not \(\sigma\)-finite for the Hausdorff measure due to Besicovitch.

We first recall the definition of the Hausdorff measure with dimension given by a continuous function \(h\).
Henceforth, we assume that \(h : [0, +\infty) \to \R\) is a continuous function such that
\begin{enumerate}[\((a)\)]
\item \(h\) is nondecreasing,
\item \(h(0) = 0\) and for every \(t > 0\), \(h(t) > 0\).
\end{enumerate}

Given a set \(A \subset \R^N\), define for \(0 < \delta \le +\infty\) the Hausdorff outer measures of dimension \(h\),
\[
\Lambda_\delta^h(A) = \inf{\bigg\{ \sum_{n = 0}^\infty h(r_n) : \textstyle A \subset \bigcup\limits_{n = 0}^\infty B(x_n; r_n) \ \text{and}\ 0 \le r_n \le \delta,\ \forall n \in \N \bigg\}}.
\]
The outer measure \(\Lambda_\infty^h(A)\) is usually called the Hausdorff content of \(A\). 

The Hausdorff measure of \(A\) of dimension \(h\) is defined as the limit
\[
\Lambda^h(A) = \lim_{\delta \to 0}{\Lambda^h_\delta(A)}.
\]
For instance, if \(s > 0\) and \(h\) is defined for \(t > 0\) by 
\[
h(t) = \omega_s t^s,
\]
where 
\(
\omega_s = \frac{\pi^{\frac{s}{2}}}{\Gamma(\frac{s}{2} + 1)}
\),
then \(\Lambda^h\) is the Hausdorff measure of dimension \(s\), denoted by  \(\cH^s\).

We have adopted Hausdorff's original definition~\cite{Hausdorff}*{Definition~1} of Hausdorff measures.
In order to emphasize that only balls are involved in the covering of the set \(A\), many authors refer to this notion as the \emph{spherical} Hausdorff measure.

\medskip

An important property related to the Hausdorff measure \(\Lambda^h\) is given by Frostman's lemma~\citelist{\cite{Frostman}*{No.~47} \cite{Carleson}*{Chapter~II, Theorem~1} \cite{Mattila}*{Theorem~8.17}}:

\begin{proposition}
\label{propositionFrostman}
Let \(A \subset \R^N\) be a Borel set.
If \(\Lambda^h(A) > 0\), then there exists a finite positive measure \(\mu\)  supported in \(A\) such that for every \(x \in \R^N\) and for every \(r > 0\),
\[
\mu(B(x; r)) \le h(r).
\]
\end{proposition}

We also need the following result of Besicovitch~\citelist{\cite{Besicovitch}*{Theorem~7} \cite{Rogers_62}*{Theorem~1}}:

\begin{proposition}
\label{propositionBesicovitch}
Let \(A \subset \R^N\) be a Borel set. 
If \(A\) is not \(\sigma\)-finite for the Hausdorff measure \(\cH^{N-1}\), then there exists an increasing continuous function \(h : [0, +\infty) \to \R\) such that 
\begin{enumerate}[$(i)$]
\item \(h(0) = 0\) and for every \(t > 0\), \(0 < h(t) \le t^{N-1}\),
\item \(\displaystyle \lim_{t \to 0}{\frac{h(t)}{t^{N-1}}} = 0\),
\item \(A\) is not \(\sigma\)-finite with respect to the Hausdorff measure \(\Lambda^h\).
\end{enumerate}
\end{proposition}

The result of Besicovitch actually concerns any Hausdorff measure \(\Lambda^{\tilde h}\) instead of \(\cH^{N-1}\). 
Although quite different in nature, this type of property is reminiscent of a result of de la Vallée Poussin which asserts that for any \(L^1\) function \(f\) there exists a superlinear continuous function \(\Phi : \R \to \R\) such that \(\Phi \circ f\) is also an \(L^1\) function~\cite{Del:1915}*{Remarque~23}.

\begin{proof}[Proof of Theorem~\ref{theoremRemovableSingularity}]
The reverse implication ``\(\Leftarrow\)'' is established in \cite{deValeriola_Moonens}*{Theorem~12}. 
In order to prove the direct implication ``\(\Rightarrow\)'', let \(S \subset \R^N\) be a closed set which is not \(\sigma\)-finite for the Hausdorff measure \(\cH^{N-1}\). 
By Besicovitch's result concerning sets which are not \(\sigma\)-finite (Proposition~\ref{propositionBesicovitch}) applied to the set \(S\), there exists an increasing continuous function \(h \) satisfying properties \((i)\)--\((iii)\) above.
In particular, 
\[
\Lambda^h(S) > 0.
\]
By Frostman's lemma (Proposition~\ref{propositionFrostman}), there exists a finite positive measure \(\mu\) supported in \(S\) such that for every \(x \in \R^N\) and for every \(r > 0\),
\[
\mu(B(x; r)) \le h(r).
\]
By  Proposition~\ref{propositionBesicovitch}~\((ii)\), the measure \(\mu\)
satisfies the condition of Phuc and Torres \eqref{eqPhucTorresContinuous}, hence \(\mu\) can be written as the divergence of a continuous vector field \(V\); see \cite{Phuc_Torres}*{Theorem~4.5}. 
We deduce that the \(C^0\) removable singularity property is not satisfied by the set \(S\).
\end{proof}


\section{Uniform Hausforff estimates}
\label{SectionUniformHausdorffEstimates}

The goal of this section is to obtain comparison inequalities between measures and the Hausdorff outer measures \(\Lambda_\delta^h\).
The type of result we have in mind is the following:

\begin{proposition}
\label{propositionStrongEstimateHausdorffMeasure}
Let \(\mu\) be a finite nonnegative measure in \(\R^N\). 
If for every Borel set \(A \subset \R^N\),
\[
\Lambda^h(A) = 0 \quad \text{implies} \quad \mu(A) = 0,
\] 
then for every \(\epsilon > 0\) there exist \(c > 0\) and a Borel set \(E \subset \R^N\) such that
\begin{enumerate}[\((i)\)]
\item \(\mu\lfloor_{E} {}\le c \Lambda^h_{\infty}\), 
\item \(\mu(\R^N \setminus E) \le \epsilon\). 
\end{enumerate}
\end{proposition}

In the statement above, we denote by \(\mu\lfloor_{E}\) the measure defined for every Borel set \(A \subset \R^N\) by
\[
\mu\lfloor_{E}(A) = \mu( A \cap E ).
\]
Throughout this paper, an estimate of the type 
\[
\mu\lfloor_{E} {}\le c \Lambda^h_{\infty}
\]
is to be understood in the sense of measures: for every Borel set \(A \subset \R^N\),
\[
\mu\lfloor_{E}(A) {}\le c \Lambda^h_{\infty}(A), 
\]
Applying this inequality on a ball \(A = B(x; r)\), we deduce the following density estimate
\[
\mu(B(x; r) \cap E) {}\le c \Lambda^h_{\infty}(B(x; r)) \le ch(r).
\]

\medskip

The following analogue of Proposition~\ref{propositionStrongEstimateHausdorffMeasure}
will be used in the proof of Theorem~\ref{theoremL1Perturbation}:

\begin{proposition}
\label{propositionEpsilonApproximation}
Let \(\mu\) be a finite nonnegative measure in \(\R^N\). 
If for every Borel set \(A \subset \R^N\), 
\[
\Lambda^h(A) < +\infty \quad \text{implies} \quad \mu(A) = 0,
\]
then for every \(\epsilon > 0\), there exists a Borel set \(E \subset \R^N\) such that
\begin{enumerate}[\((i)\)]
\item for every \(c > 0\) there exists \(\delta > 0\) such that 
\(\mu\lfloor_{E} {}\le c \Lambda^h_{\delta}\),
\item \(\mu(\R^N \setminus E) \le \epsilon\).
\end{enumerate}
\end{proposition}

Although Proposition~\ref{propositionStrongEstimateHausdorffMeasure} and Proposition~\ref{propositionEpsilonApproximation} look similar, the proof of the latter is more involved and requires an additional argument in order to compensate the lack of additivity of the Hausdorff outer measures \(\Lambda_\delta^h\).
In both cases, we do not rely on covering arguments in \(\R^N\) so the proofs are also valid in metric spaces; concerning Proposition~\ref{propositionEpsilonApproximation}, we need the inner regularity of the measure \(\mu\).

A variant of both propositions for finite measures satisfying the inequality \(\mu \le \alpha \Lambda^h\) for some \(\alpha > 0\) may be found in \cite{Ponce}*{Section~8.3}.
This type of estimate arises for instance in the study of elliptic PDEs with exponential nonlinearity involving measure data~\citelist{\cite{BLOP:05} \cite{Ponce}*{Chapter~8}}.

\medskip

The results of this section are based on the following:

\begin{lemma}
\label{lemmaHahnDecomposition}
Let \(\mu\) be a finite nonnegative measure in \(\R^N\).
For every nonnegative outer measure \(T\), there exists a Borel set \(E \subset \R^N\) such that
\[
\mu\lfloor_{E} {}\le T
\quad 
\text{and}
\quad
T(\R^N \setminus E) \le \mu(\R^N \setminus E).
\]
\end{lemma}

We recall the a nonnegative outer measure \(T\) is a set function with values into \([0, +\infty]\) such that
\begin{enumerate}[\((a)\)]
\item \(T(\emptyset) = 0\),
\item if \(A \subset B\), then \(T(A) \le T(B)\),
\item for every sequence \((A_n)_{n \in \N}\) of Borel sets, \(T\big(\bigcup\limits_{k=0}^\infty A_k\big) \le \sum\limits_{k=0}^\infty T(A_k)\).
\end{enumerate}

This lemma is inspired from \cite{BLOP:05}*{Lemma~2};
the outer measure \(T\) we have in mind is the Hausdorff outer measure \(\Lambda^h_{\delta}\).
When \(T\) is a finite measure, this lemma follows from the classical Hahn decomposition theorem \cite{Fol:99}*{Theorem~3.3} applied to the measure \(\mu - T\), in which  case the set \(E\) may be chosen so that
\[
\mu\lfloor_{E} {}\le T
\quad 
\text{and}
\quad
T\lfloor_{\R^N \setminus E} {}\le \mu.
\]

The main idea of the proof is that if the inequality \(\mu \le T\) does not hold on every Borel set, then there exists some Borel set \(F \subset \R^N\) such that
\(T(F) < \mu(F)\) and we try to choose \(F\) so that \(\mu(F)\) is as large as possible.
Since \(\mu\) is a finite measure, we eventually exhaust the part of \(\mu\) that prevents the inequality \(\mu \le T\) to hold.

\begin{proof}[Proof of Lemma~\ref{lemmaHahnDecomposition}]
Let \(0 < \theta < 1\).
By induction, there exists a sequence \((F_n)_{n \in \N}\) of disjoint Borel sets of \(\R^N\) such that
\begin{enumerate}[\((a)\)]
\item for every \(n \in \N\), \(T(F_n) \le \mu(F_n)\),
\item for every \(n \in \N_*\), \(\mu(F_n) \ge \theta \epsilon_n\), where
\[
\textstyle \epsilon_n = \sup{\Big\{ \mu(F) : F \subset \R^N \setminus \bigcup\limits_{k=0}^{n-1} F_k \ \text{and} \ T(F) \le \mu(F) \Big\}}.
\]
\end{enumerate}
By subadditivity of \(T\) and by additivity of \(\mu\),
\[
\textstyle T\big(\bigcup\limits_{k=0}^\infty F_k\big) 
\le \displaystyle  \sum_{k=0}^\infty T(F_k) 
\le \sum_{k=0}^\infty \mu(F_k) 
= \textstyle \mu\big(\bigcup\limits_{k=0}^\infty F_k\big).
\]

We claim that
\[
\mu\lfloor_{\R^N \setminus \bigcup\limits_{k=0}^\infty F_k} {}\le T.
\]
Assume by contradiction that this inequality is not true.
Then, there exists a Borel set \(C \subset \R^N\) such that
\[
\textstyle T(C) < \mu\big(C \setminus \bigcup\limits_{k=0}^\infty F_k\big).
\]
Let
\[
\textstyle D = C \setminus \bigcup\limits_{k=0}^\infty F_k.
\]
By monotonicity of \(T\), we have
\[
T(D) \le T(C) < \mu(D).
\]
In particular, \(\mu(D) > 0\).
Since \(D\) is an admissible set in the definition of the numbers \(\epsilon_n\),
 for every \(n \in \N\) we have
\[
\mu(D) \le \epsilon_n.
\]
This is not possible since
\[
\theta \sum_{k=1}^\infty \epsilon_k 
\le \sum_{k=1}^\infty \mu(F_{k}) = \textstyle \mu\big(\bigcup\limits_{k=1}^\infty F_{k}\big) \le \mu(\R^N) < +\infty.
\]
In particular, the sequence \((\epsilon_n)_{n \in \N}\) converges to \(0\), but this contradicts the fact that \((\epsilon_n)_{n \in \N}\) is bounded from below by \(\mu(D)\).

We have the conclusion of the lemma by choosing 
\[
\textstyle E = \R^N \setminus \bigcup\limits_{k=0}^\infty F_k.
\qedhere
\]
\end{proof}

The second ingredient in the proof of Proposition~\ref{propositionStrongEstimateHausdorffMeasure} gives a quantitative information of the absolute continuity with respect to the Hausdorff measure \(\Lambda^h\) in terms of the Hausdorff content \(\Lambda_\infty^h\). 
The proof uses the same strategy as in the proof of the usual absolute continuity of measures.
It relies on the fact that for any set \(A \subset \R^N\), \(\Lambda^h(A) = 0\) if and only if \(\Lambda_\infty^h(A) = 0\).

\begin{lemma}
\label{lemmaAbsoluteContinuityHausdorff}
Let \(\mu\) be a finite nonnegative measure in \(\R^N\). 
If for every Borel set \(A \subset \R^N\),
\[
\Lambda^h(A) = 0 \quad \text{implies} \quad \mu(A) = 0,
\]
then for every \(\epsilon > 0\) there exists \(\eta > 0\) such that for every Borel set \(A \subset \R^N\), 
\[
\Lambda^h_\infty(A) \le \eta \quad \text{implies} \quad {\mu}(A) \le \epsilon.
\]
\end{lemma}

\begin{proof}
Assume by contradiction that there exists \(\epsilon > 0\) such that for every \(\eta > 0\) there exists a Borel set \(A \subset \R^N\) such that \(\Lambda^h_\infty(A) \le \eta\) and \({\mu}(A) > \epsilon\).

Given a sequence  \((\eta_n)_{n \in \N}\) of positive numbers, 
for every \(n \in \N\) we take a Borel set \(A_n \subset \R^N\) such that
\[
\Lambda^h_\infty(A_n) \le \eta_n
\quad \text{and} \quad 
{\mu}(A_n) > \epsilon.
\]
By subadditivity of \(\Lambda_\infty^h\) we have for every \(n \in \N\),
\[
\textstyle \Lambda^h_\infty \big(\bigcup\limits_{k=n}^\infty A_k \big) 
\displaystyle
\le \sum_{k=n}^\infty \Lambda^h_\infty(A_k) \le \sum_{k=n}^\infty \eta_k
\]
and by monotonicity of \(\mu\),
\[
\textstyle {\mu}(\bigcup\limits_{k=n}^\infty A_k) 
\ge {\mu}(A_n) > \epsilon.
\]
Take
\[
\textstyle B = \bigcap\limits_{n=0}^\infty \bigcup\limits_{k=n}^\infty A_k.
\]
Choosing the sequence  \((\eta_n)_{n \in \N}\) such that the series \(\sum\limits_{k=0}^\infty \eta_k\) converges,
we deduce from the above that 
\[
\Lambda^h_\infty(B) = 0
\quad \text{and} \quad 
{\mu}(B) \ge \epsilon.
\] 
Since \(\Lambda^h_\infty(B) = 0\) is equivalent to \(\Lambda^h(B) = 0\), we get a contradiction with the assumption on the measure \(\mu\).
\end{proof}

\begin{proof}[Proof of Proposition~\ref{propositionStrongEstimateHausdorffMeasure}]
Given \(c > 0\), let \(E_c \subset \R^N\) be the set given by Lemma~\ref{lemmaHahnDecomposition} with measure \({\mu}\) and outer measure \(T = c \Lambda_\infty^h\). 
Thus,
\[
\mu\lfloor_{E_c} {}\le c \Lambda^h_{\infty}
\] 
and
\[
c \Lambda^h_\infty(\R^N \setminus E_c) \le \mu(\R^N \setminus E_c) \le \mu(\R^N).
\]
In particular,
\[
\lim_{c \to +\infty}{\Lambda^h_\infty(\R^N \setminus E_c)} = 0.
\]
By the property of absolute continuity of \(\mu\) with respect to \(\Lambda_\infty^h\) (Lemma~\ref{lemmaAbsoluteContinuityHausdorff}), the conclusion of the proposition follows.
\end{proof}

In the proof of Proposition~\ref{propositionEpsilonApproximation}, we  bypass the lack of additivity of the outer Hausdorff measures \(\Lambda_\delta^h\) using the following

\begin{lemma}
\label{lemmaHausdorffOuterMeasureAdditivity}
Let \(\nu \in \cM(\R^N)\) be a finite nonnegative measure in \(\R^N\), let \(\delta > 0\) and let \(F_1, \dots, F_n\) be disjoint Borel subsets of \(\R^N\).
If for every \(k \in \{1, \dots, n\}\),
\[
\nu\lfloor_{F_k} {}\le \Lambda_\delta^h,
\]
then for every \(\epsilon > 0\), there exist \(0 < \underline{\delta} {}\le \delta\) and a Borel set \(F \subset \bigcup\limits_{k=1}^n F_k\) such that
\[
\nu\lfloor_{F} {}\le \Lambda_{\underline{\delta}}^h
\quad
\text{and}
\quad
\textstyle\nu\big( \bigcup\limits_{k = 1}^n F_k \setminus E \big) \le \epsilon.
\]
\end{lemma}

\begin{proof}
For each \(i \in \{1, \dots, n\}\), let \(K_i \subset F_i\) be a compact subset.
For every Borel set \(A \subset \R^N\),
\[
\textstyle \nu\lfloor_{\bigcup\limits_{i = 1}^{n} K_i}(A) 
= \displaystyle \sum_{i = 1}^{n} \nu(A \cap K_i)
\le \sum_{i = 1}^{n} \Lambda_{\delta}^h(A \cap K_i).
\]
Let \(0 < \underline{\delta} {}\le \delta\) be such that for every \(i, j \in \{1, \dotsc, n\}\), if \(i \ne j\), then \(d(K_i, K_j) \ge 2\underline{\delta}\). 
In particular,
\[
d(A \cap K_i, A \cap K_j) \ge 2\underline{\delta}.
\]
By the metric additivity of the outer Hausdorff measure \(\Lambda_\delta^h\) \citelist{ \cite{Rogers_70}*{Theorem~16} \cite{Fol:99}*{proof of Proposition~11.17}},
\[
\sum_{i=1}^n \Lambda^h_{\underline{\delta}}(A \cap K_i) 
= \textstyle \Lambda^h_{\underline{\delta}}\big(\bigcup\limits_{i=1}^n (A \cap  K_i)\big). \]
We deduce that for every Borel set \(A \subset \R^N\),
\[
\begin{split}
\textstyle \nu\lfloor_{\bigcup\limits_{i = 1}^{n} K_i}(A) 
& \le \displaystyle \sum_{i = 1}^{n} \Lambda_{\delta}^h(A \cap K_i)\\
& \le \displaystyle \sum_{i = 1}^{n} \Lambda_{\underline{\delta}}^h(A \cap K_i)
= \textstyle \Lambda^h_{\underline{\delta}}\big(\bigcup\limits_{i=1}^n (A \cap  K_i)\big)
\le \Lambda^h_{\underline{\delta}}(A). 
\end{split}
\]
Thus, the set 
\[
F = \textstyle \bigcup\limits_{i = 1}^{n} K_i
\]
satisfies the first property of the statement.

We now show how to choose the compact sets \(K_i\) in order to have the second property.
Since 
\[
\textstyle 
\big(\bigcup\limits_{i=1}^{n} F_i \big) \setminus \big(\bigcup\limits_{i=1}^{n} K_i \big)
=  
\bigcup\limits_{i=1}^{n} (F_i \setminus K_i),
\]
by additivity of the measure \(\mu\),
\[
\textstyle 
\mu\big( \big(\bigcup\limits_{i=1}^{n} F_i \big) \setminus \big(\bigcup\limits_{i=1}^{n} K_i \big) \big) 
=
\displaystyle 
\sum\limits_{i=1}^{n} 
\textstyle
\mu( F_i \setminus K_i).
\]
By inner regularity of the measure \(\mu\), we may choose the compact set \(K_i \subset F_i\) such that
\[
\mu(F_i \setminus K_i) \le \frac{\epsilon}{n}.
\]
Thus,
\[
\textstyle 
\mu\big( \big(\bigcup\limits_{i=1}^{n} F_i\big) \setminus \big( \bigcup\limits_{i=1}^{n} K_i \big)\big) 
\le n \dfrac{\epsilon}{n} = \epsilon.
\]
This proves the lemma.
\end{proof}

\begin{proof}[Proof of Proposition~\ref{propositionEpsilonApproximation}]
We begin by proving the existence of a Borel set \(E \subset \R^N\) depending on \(c > 0\):

\begin{Claim}
For every \(\epsilon > 0\) and for every \(c > 0\), there exist a Borel set \(E \subset \R^N\) and \(\delta > 0\) satisfying properties \((i)\)--\((ii)\).
\end{Claim}

\begin{proof}[Proof of the claim]
Let \((\delta_n)_{n \in \N}\) be a sequence of positive numbers. 
Given \(c > 0\), we construct a sequence of Borel sets \((F_n)_{n \in \N}\) in \(\R^N\) such that
\begin{enumerate}[\((a)\)]
\item \(\mu\lfloor_{F_0} {}\le c\Lambda^h_{\delta_{0}}\),
\item for every \(n \in \N_*\), 
\(\mu\lfloor_{F_{n} \setminus \bigcup\limits_{k=0}^{n-1} F_k} {}\le  c\Lambda^h_{\delta_{n}}\),
\item for every \(n \in \N\), 
\(
c\Lambda^h_{\delta_{n}}(\R^N \setminus F_n) \le \mu\big(\R^N \setminus \bigcup\limits_{k=0}^{n} F_k\big).
\)
\end{enumerate}

We proceed by induction on \(n \in \N\).
Let \(F_0 \subset \R^N\) be a Borel set satisfying the conclusion of Lemma~\ref{lemmaHahnDecomposition} with \(T = \Lambda_{\delta_0}^h\). 
Given \(n \in \N_*\) and Borel sets \(F_0, \dotsc, F_{n-1}\), we apply Lemma~\ref{lemmaHahnDecomposition} with measure \(\mu\lfloor_{\R^N \setminus \bigcup\limits_{k=0}^{n-1} F_k}\) and outer measure \(T = c \Lambda_{\delta_n}^h\).
We then obtain a Borel set \(F_{n} \subset \R^N\)
such that
\[
\mu\lfloor_{F_{n} \setminus \bigcup\limits_{k=0}^{n-1} F_k} {}\le c\Lambda^h_{\delta_{n}}
\]
and
\[
\textstyle c\Lambda^h_{\delta_{n}}(\R^N \setminus F_n) 
\le \mu\big(\R^N \setminus \bigcup\limits_{k=0}^{n} F_k\big).
\]
This sequence \((F_n)_{n \in \N}\) satisfies properties \((a)\)--\((c)\).
\medskip

We now observe that
\[
\lim_{n \to \infty}{\mu\big(\R^N \setminus \textstyle\bigcup\limits_{k=0}^{n} F_k\big)} = \mu(C),
\]
where
\[
\textstyle C = \R^N \setminus \bigcup\limits_{k=0}^{\infty} F_k.
\]
By monotonicity of the outer measure \(\Lambda^h_{\delta_{n}}\) and by property \((c)\) above,
\[
\textstyle c\Lambda^h_{\delta_{n}}(C) 
\le c\Lambda^h_{\delta_{n}}(\R^N \setminus F_n) 
\le \mu\big(\R^N \setminus \bigcup\limits_{k=0}^{n} F_k\big).
\]
Choosing a sequence \((\delta_n)_{n \in \N}\) converging to zero, as \(n\) tends to infinity we get
\[
c\Lambda^h(C) \le \mu(C).
\]
In particular, the Hausdorff measure \(c\Lambda^h(C)\) is finite and by assumption on the measure \(\mu\),
\[
\mu(C) = 0.
\]
Thus,
\begin{equation}
\label{eqLimit}
\lim_{n \to \infty}{\mu\big(\R^N \setminus \textstyle\bigcup\limits_{k=0}^{n} F_k\big)} = \mu(C) = 0.
\end{equation}

Given a Borel set \(E \subset \textstyle\bigcup\limits_{k=0}^{n} F_k\),
by additivity of the measure \(\mu\),
\[
\mu(\R^N \setminus E) 
= \mu\big(\R^N \setminus \textstyle\bigcup\limits_{k=0}^{n} F_k\big) + 	\mu\big(\textstyle\bigcup\limits_{k=0}^{n} F_k \setminus E \big).
\]
Given \(\epsilon > 0\), by the limit \eqref{eqLimit} above there exists \(n \in \N\) such that
\[
\mu\big(\R^N \setminus \textstyle\bigcup\limits_{k=0}^{n} F_k\big) 
\le \dfrac{\epsilon}{2}.
\]
By the property of weak additivity of the Hausdorff outer measures \(\Lambda_\delta^h\) (Lemma~\ref{lemmaHausdorffOuterMeasureAdditivity}) applied to the sets \(F_0, F_1\setminus F_0, \dots, F_n \setminus \bigcup\limits_{k=0}^{n-1} F_k\), there exist
\[
0 < \underline{\delta} {}\le \min{\{\delta_0, \dots, \delta_n\}}
\quad
\text{and}
\quad
\textstyle E \subset \bigcup\limits_{k=0}^{n} F_k
\]
such that
\[
\mu\lfloor_E \le \Lambda_{\underline{\delta}}^h 
\quad \text{and} \quad
\mu\big(\textstyle\bigcup\limits_{k=0}^{n} F_k \setminus E\big) \le \dfrac{\epsilon}{2}.
\]
Thus,
\[
\mu(\R^N \setminus E)  \le \frac{\epsilon}{2} + \frac{\epsilon}{2} = \epsilon.
\]
This concludes the proof of the claim.
\end{proof}

Given two sequences \((\epsilon_n)_{n \in \N}\) and \((c_n)_{n \in \N}\) of positive numbers to be chosen below, for every \(n \in \N\) by the previous claim there exist a Borel set \(E_n \subset \R^N\) and \(\delta_n > 0\) such that
\[
\mu\lfloor_{E_n} {}\le c_n \Lambda_{\delta_n}^h
\quad \text{and} \quad
\mu(\R^N \setminus E_n) \le \epsilon_n.
\]
Let
\[
E = \textstyle \bigcap\limits_{k=0}^\infty E_k.
\]
By monotonicity of \(\mu\) we have for every \(n \in \N\),
\[
\mu\lfloor_{E} {}\le \mu\lfloor_{E_n} {}\le c_n \Lambda_{\delta_n}^h.
\]
Choosing a sequence \((c_n)_{n \in \N}\) converging to zero, then for every \(c > 0\), there exists \(n \in \N\) such that \(c_n \le c\) and we have
\[
\mu\lfloor_{E} {}\le c_n \Lambda_{\delta_n}^h \le c \Lambda_{\delta_n}^h.
\]
Thus, property \((i)\) holds with \(\delta = \delta_n\).

By subadditivity of the measure \(\mu\),
\[
\mu(\R^N \setminus E) \le \sum_{k = 0}^\infty \mu(\R^N \setminus E_k) \le \sum_{k = 0}^\infty \epsilon_k.
\]
Choosing the sequence \((\epsilon_n)_{n \in \N}\) such that the series \(\sum\limits_{k = 0}^\infty \epsilon_k\) converges and is bounded from above by \(\epsilon\), we deduce property~\((ii)\). 
The proof of the proposition is complete.
\end{proof}

\medskip

The next corollary is the counterpart of Lemma~\ref{lemmaAbsoluteContinuityHausdorff} for measures which do not charge sets of finite Hausdorff measure \(\Lambda^h\).
We have no direct proof of it as in the case of Lemma~\ref{lemmaAbsoluteContinuityHausdorff}. 
If we could prove directly Corollary~\ref{corollaryAbsoluteContinuityFiniteHausdorff}, then we would have a simpler proof of Proposition~\ref{propositionEpsilonApproximation} in the lines of the proof of Proposition~\ref{propositionStrongEstimateHausdorffMeasure}.

\begin{corollary}
\label{corollaryAbsoluteContinuityFiniteHausdorff}
Let \(\mu\) be a finite nonnegative measure in \(\R^N\). 
If for every Borel set \(A \subset \R^N\), 
\[
\Lambda^h(A) < +\infty \quad \text{implies} \quad \mu(A) = 0,
\]
then for every \(M \ge 0\) and for every \(\epsilon > 0\), there exists \(\delta > 0\) such that for every Borel set \(A \subset \R^N\), 
\[
\Lambda^h_\delta(A) \le M \quad \text{implies} \quad {\mu}(A) \le \epsilon.
\]
\end{corollary}

\begin{proof}
Given a Borel set \(E \subset \R^N\), for every Borel set \(A \subset \R^N\),
\[
{\mu}(A) = {\mu}(A \cap E) + {\mu}(A \setminus E) 
\le {\mu}(A \cap E) + {\mu}(\R^N \setminus E).
\]
Given \(\epsilon > 0\), let \(E \subset \R^N\) be a Borel set satisfying the conclusion of Proposition~\ref{propositionEpsilonApproximation}.
Thus, for every \(c > 0\) there exists \(\delta > 0\) such that
\[
{\mu}(A) \le c \Lambda^h_{\delta} (A) + \frac{\epsilon}{2}.
\]
For every \(M \ge 0\), choose \(c > 0\) such that \(cM \le \frac{\epsilon}{2}\).
It follows that if \(\Lambda^h_{\delta}(A) \le M\), then
\[
{\mu}(A) \le cM + \frac{\epsilon}{2} \le \epsilon.
\qedhere
\]
\end{proof}


\section{Proof of Theorem~\ref{theoremL1Perturbation}}
\label{Section4}

One of the main tools in the proof of Theorem~\ref{theoremL1Perturbation} is Proposition~\ref{theoremStrongConvergenceMeasures} which relates conditions \eqref{eqPhucTorresContinuous} and \eqref{eqConditionSigmaFinite}.
The counterpart for the Hausdorff measure \(\Lambda^h\) is the following:

\begin{proposition}
\label{theoremStrongConvergenceMeasuresGeneralCase}
Let \(\mu\) be a finite nonnegative measure in \(\R^N\). 
If for every Borel set \(A \subset \R^N\),
\[
\Lambda^h(A) < +\infty \quad \text{implies} \quad \mu(A) = 0,
\]
then there exists a nondecreasing sequence of finite nonnegative measures \((\mu_n)_{n \in \N}\) in \(\R^N\) such that for every \(n \in \N\),
\[
\lim_{\delta \to 0}{\sup_{\substack{x \in \R^N\\r \le \delta}}{\frac{\mu_n(B(x; r))}{h(r)}}} = 0
\]
and
\[
\lim_{n \to \infty}{\norm{\mu_n - \mu}_{\cM(\R^N)}} = 0.
\]
\end{proposition}

In the proof of Theorem~\ref{theoremL1Perturbation}, it will be more convenient to use the counterpart of this approximation stated as a decomposition of the measure \(\mu\) in terms of a series of measures.

\begin{proposition}
\label{propositionStrongConvergenceMeasuresGeneralCaseSeries}
Let \(\mu\) be a finite nonnegative measure in \(\R^N\). 
If for every Borel set \(A \subset \R^N\),
\[
\Lambda^h(A) < \infty
\quad \text{implies} \quad
\mu(A) = 0,
\]
then there exists a sequence \((\nu_n)_{n \in \N}\) of finite nonnegative measures in \(\R^N\) such that for every \(n \in \N\),
\[
\lim_{\delta \to 0}{\sup_{\substack{x \in \R^N\\r \le \delta}}{\frac{\nu_n(B(x; r))}{h(r)}}} = 0
\]
and for every Borel set \(A \subset \R^N\),
\[
\mu(A) = \sum_{k=0}^\infty {\nu_k(A)}.
\]
\end{proposition}

\begin{proof}
Let \((\epsilon_n)_{n \in \N}\) be a sequence of positive numbers.
By Proposition~\ref{propositionEpsilonApproximation}, there exists a Borel set \(E_0 \subset \R^N\) such that
\begin{enumerate}[\((a)\)]
\item for every \(c > 0\) there exists \(\delta > 0\) such that \(\mu\lfloor_{E_0} {}\le c \Lambda_\delta^h\),
\item \(\mu(\R^N \setminus E_0) \le \epsilon_0\).
\end{enumerate}
We proceed by induction. 
Given \(n \in \N_*\) and Borel sets \(E_0, \dotsc, E_{n-1} \subset \R^N\), 
we apply Proposition~\ref{propositionEpsilonApproximation} with measure \(\mu\lfloor_{\R^N \setminus \bigcup\limits_{k=0}^{n-1} E_k}\) and parameter \(\epsilon_n\).
We obtain a Borel set \(E_n \subset \R^N\) such that
\begin{enumerate}[\((a')\)]
\item for every \(c > 0\) there exists \(\delta > 0\) such that \(\mu\lfloor_{E_n \setminus \bigcup\limits_{k=0}^{n-1} E_k} {}\le c \Lambda_\delta^h\),
\item \(\mu(\R^N \setminus \bigcup\limits_{k=0}^{n} E_k) \le \epsilon_n\).
\end{enumerate}

Let \(F_0 = E_0\) and for \(n \in \N_*\),
\[
\textstyle F_n = E_n \setminus \bigcup\limits_{k=0}^{n-1} E_k.
\]
The sets \(F_0, F_1, \ldots\) are disjoint.
By additivity of the measure \(\mu\), for every Borel set \(A \subset \R^N\),
\[
\mu(A) - \sum_{k=0}^n \mu(A \cap F_k) 
=
\textstyle \mu\big(A \setminus \bigcup\limits_{k=0}^{n} E_k\big) 
\le \epsilon_n.
\]
Choosing a sequence \((\epsilon_n)_{n \in \N}\) converging to zero, we deduce that 
\[
\mu(A) = \sum_{k=0}^\infty {\mu(A \cap F_k)}.
\]
The conclusion follows by choosing for every \(n \in \N\),
\[
\nu_n = \mu \lfloor_{F_n}.
\qedhere
\]
\end{proof}

We prove Theorem~\ref{theoremL1Perturbation} using a strategy of Boccardo, Gallouët and Orsina \cite{BocGalOrs:96}*{proof of Theorem~2.1} originally used to show that every finite measure which is diffuse with respect to the \(W^{1, p}\) capacity for some \(1 < p < +\infty\) belongs to \((W_0^{1, p})' + L^1\).

\begin{proof}[Proof of Theorem~\ref{theoremL1Perturbation}]
Let \(\mu\) be a finite measure in \(\R^N\) such that for every Borel set \(A \subset \R^N\), \(\cH^{N-1}(A) < +\infty\) implies \(\mu(A) = 0\).
By the Hahn decomposition theorem, the same property is still satisfied by the positive and the negative parts of \(\mu\). 
We may thus assume throughout the proof that \(\mu\) is a nonnegative measure.

Let \((\nu_n)_{n \in \N}\) be a sequence of measures given by the previous proposition.
We decompose the measure \(\mu\) as a series of finite measures 
\[
\mu =  \sum\limits_{k=0}^\infty \nu_k.
\]
By the characterization of nonnegative measures for which the equation \eqref{eqDiv} has a continuous solution \cite{Phuc_Torres}*{Theorem~4.5}, for each \(n \in \N\) there exists a continuous vector field \(W_n : \R^N \to \R^N\) such that 
\[
\nu_n = \div{W_n} \quad \text{in \(\R^N\).}
\]
Given a smooth mollifier \(\rho\) and \(\delta > 0\), let \(\rho_{\delta} : \R^N \to \R\) be the function defined for \(x \in \R^N\) by \(\rho_{\delta}(x) = \frac{1}{\delta^N} \rho(\frac{x}{\delta})\).
We write
\[
\begin{split}
\nu_n
& =  (\nu_n - \rho_{\delta} * \nu_n) + \rho_{\delta} * \nu_n\\
& =  \div{(W_n - \rho_{\delta} * W_n)} + \rho_{\delta} * \nu_n.
\end{split}
\]

Let \((\delta_n)_{n \in \N}\) be a sequence of positive numbers and \(j \in \N\) to be chosen below.
We write the measure \(\mu\) --- at least formally for the moment --- as
\begin{equation}
\label{eqDecompositionMeasure}
\begin{split}
\mu
& = \sum\limits_{k=0}^j \nu_k +  \sum\limits_{k= j + 1}^\infty \nu_k\\
& = \div{\Big(\sum\limits_{k=0}^j W_k  + \sum_{k=j+1}^\infty (W_k - \rho_{\delta_k} * W_k)\Big)}
+  \sum_{k = j+1}^\infty \rho_{\delta_k} * \nu_k.
\end{split}
\end{equation}

The last series defines an \(L^1\) function regardless of the choice of the sequence \((\delta_n)_{n \in \N}\). 
Indeed, for every \(n \in \N\), by Fubini's theorem we have
\[
\norm{\rho_{\delta_n} * \nu_n}_{L^1(\R^N)} 
= \norm{\nu_n}_{\cM(\R^N)} = \nu_n(\R^N).
\]
Thus, for every \(j \in \N\) the series 
\[
\sum_{k = j + 1}^\infty \rho_{\delta_k} * \nu_k
\]  
converges in \(L^1(\R^N)\).

Since for each \(n \in \N\), \(W_n\) is continuous in \(\R^N\), the family \((\rho_\delta * W_n)_{\delta > 0}\) converges to \(W_n\) uniformly to \(W_n\) on bounded subsets of \(\R^N\). 
We choose the sequence \((\delta_n)_{n \in \N}\) such that for every \(n \in \N\),
\[
\norm{W_n - \rho_{\delta_n} * W_n}_{L^\infty(B_n(0))} {}\le \alpha_n,
\]
where \((\alpha_n)_{n \in \N}\) is any sequence of positive numbers such that the series \(\sum\limits_{k=0}^\infty \alpha_k\) converges.
In particular, the series
\[
\sum_{k=j+1}^\infty (W_k - \rho_{\delta_k} * W_k)
\]
converges locally uniformly in \(\R^N\).

In view of the above, \eqref{eqDecompositionMeasure} gives a legitimate decomposition of the measure \(\mu\) of the form
\[
\mu = \div{V_j} + f_j  \quad \text{in \(\R^N\)},
\]
where the vector field \(V_j : \R^N \to \R^N\) is continuous and \(f_j \in L^1(\R^N)\) satisfies
\[
\norm{f_j}_{L^1(\R^N)} 
= \sum\limits_{k= j + 1}^\infty \norm{\rho_{\delta_k} * \nu_k}_{L^1(\R^N)} 
= \sum\limits_{k= j + 1}^\infty \nu_k(\R^N).
\]
Given \(\epsilon > 0\), we have the conclusion by choosing \(j \in \N\) such that
\[
\sum\limits_{k= j + 1}^\infty \nu_k(\R^N) \le \epsilon.
\]
The proof of the theorem is complete.
\end{proof}

We now explain how the vector field \(V\) in the statement of Theorem~\ref{theoremL1Perturbation} may be chosen to vanish at infinity.
The main ingredient is a result of De~Pauw and Torres~\cite{DePauw_Torres}*{Theorem~6.1} which states that \(\mu = \div{V}\) for some continuous vector field \(V : \R^N \to \R^N\) vanishing at infinity if and only if for every sequence \((\varphi_n)_{n \in \N}\) in \(C_c^\infty(\R^N)\) converging weakly to zero in \(L^{\frac{N}{N-1}}(\R^N)\) and such that the sequence \((D\varphi_n)_{n \in \N}\) is bounded in \(L^1(\R^N)\),
\[
\lim_{n \to \infty}{\int\limits_{\R^N} \varphi_n \dif\mu} = 0.
\]
This property holds for example if \(\mu\) is a nonnegative finite measure with compact support satisfying the strong charge condition or, equivalently, \eqref{eqPhucTorresContinuous}.

By inner regularity, every finite measure may be strongly approximated by measures having compact support via restriction of the measure to compact sets.
Hence, in the conclusion of Proposition~\ref{propositionStrongConvergenceMeasuresGeneralCaseSeries} we may assume that each measure \(\nu_n\) has compact support.
By the result of De~Pauw and Torres mentioned above, we may then write for every \(n \in \N\), \(\nu_n = \div{W_n}\) where \(W_n\) is a continuous vector field vanishing at infinity.
In the proof of Theorem~\ref{theoremL1Perturbation} we then choose \(\delta_n > 0\) such that
\[
\norm{W_n - \rho_{\delta_n} * W_n}_{L^\infty(\R^N)} {}\le \alpha_n,
\]
and the conclusion follows.


\section{Proof of Corollary~\ref{corollaryCharges}}
\label{Section5}

We first explain the main estimate we need based on a strategy from \cite{DePauw_Pfeffer}.
Let \(\mu\) be a finite measure in \(\R^N\) such that
\begin{equation}
\label{eqDecompositionL1Div}
\mu = \div{V} + f \quad \text{in \(\R^N\),}
\end{equation}
where the vector field \(V : \R^N \to \R^N\) is continuous and \(f\) is a locally integrable function in \(\R^N\).
Given \(\varphi \in C_c^\infty(\R^N)\) and a smooth vector field \(W : \R^N \to \R^N\), we write
\[
\begin{split}
\int\limits_{\R^N} \varphi \dif\mu 
& =  - \int\limits_{\R^N} V \cdot \nabla\varphi + \int\limits_{\R^N} f \varphi\\
& = \int\limits_{\R^N} (\div{W}) \varphi + \int\limits_{\R^N} (W - V) \cdot \nabla\varphi + \int\limits_{\R^N} f \varphi.
\end{split}
\]
Thus,
\begin{multline}
\label{eqMainEstimate}
\biggabs{\int\limits_{\R^N} \varphi \dif\mu}
\le \norm{\div{W}}_{L^\infty(\supp{\varphi})} \norm{\varphi}_{L^1(\R^N)} + 
\norm{W - V}_{L^\infty(\supp{\varphi})} \norm{D\varphi}_{L^1(\R^N)} \\
+ \norm{f}_{L^1(\supp{\varphi})} \norm{\varphi}_{L^\infty(\R^N)}.
\end{multline}

\begin{proof}[Proof of Corollary~\ref{corollaryCharges}]
If the measure \(\mu\) satisfies property \eqref{eqConditionSigmaFinite}, then by Theorem~\ref{theoremL1Perturbation} for every \(\epsilon > 0\), \(\mu\) can be written in the form \eqref{eqDecompositionL1Div} with 
\[
\norm{f}_{L^1(\R^N)} \le \epsilon.
\]
Let \(K \subset \R^N\) be a compact set.
Since the vector field \(V\) is continuous, by the Weierstrass approximation theorem there exists a smooth vector field \(W : \R^N \to \R^N\) such that
\[
\norm{W - V}_{L^\infty(K)} \le \epsilon.
\]
In view of estimate \eqref{eqMainEstimate}, for every \(\varphi \in C_c^\infty(\R^N)\) supported in \(K\),
\[
\biggabs{\int\limits_{\R^N} \varphi \dif\mu}
\le 
 \norm{\div{W}}_{L^\infty(K)} \norm{\varphi}_{L^1(\R^N)} + \epsilon \big( 
 \norm{D\varphi}_{L^1(\R^N)} + \norm{\varphi}_{L^\infty(\R^N)} \big).
\]
Thus, \(\mu\) satisfies the definition of charge with constant \(C = \norm{\div{W}}_{L^\infty(K)}\).
 
Conversely, if \(\mu\) is a charge, then by \cite{DePauw_Pfeffer}*{Theorem~6.2} there exist a strong charge \(\nu\) and a locally integrable function \(f\) in \(\R^N\) such that
\[
\mu = \nu + f \quad \text{in \(\R^N\).}
\]
Note that \(f\) satisfies property \eqref{eqConditionSigmaFinite} and 
 \(\nu\) is a locally finite measure in \(\R^N\). 
We need to show that \(\nu\) also satisfies property \eqref{eqConditionSigmaFinite}.
By the characterization of strong charges by De~Pauw and Pfeffer \cite{DePauw_Pfeffer}*{Theorem~4.8}, there exists a continuous vector field \(V : \R^N \to \R^N\) such that
\[
\nu = \div{V} \quad \text{in \(\R^N\).}
\]
Let \(W : \R^N \to \R^N\) be a smooth vector field to be chosen below.
By estimate \eqref{eqMainEstimate} with \(f = 0\), for every \(\varphi \in C_c^\infty(\R^N)\) we have
\[
\biggabs{\int\limits_{\R^N} \varphi \dif\nu}
\le 
 \norm{\div{W}}_{L^\infty(\supp{\varphi})} \norm{\varphi}_{L^1(\R^N)}
 + \norm{W - V}_{L^\infty(\supp{\varphi})} \norm{D\varphi}_{L^1(\R^N)}.
\]

The remaining of the proof is based on a standard covering argument.
Indeed, given a compact set \(K \subset \R^N\) such that \(\cH^{N-1}(K) <+\infty\), let \((\varphi_n)_{n \in \N}\) be a sequence in \(C_c^\infty(\R^N)\) such that
\begin{enumerate}[\((a)\)]
\item \((\varphi_n)_{n \in \N}\) is bounded in \(L^\infty(\R^N)\),
\item \((D\varphi_n)_{n \in \N}\) is bounded in \(L^1(\R^N)\),
\item \((\varphi_n)_{n \in \N}\) converges pointwisely to \(0\) in \(\R^N \setminus K\),
\item for every \(n \in \N\), \(\varphi_n = 1\) in \(K\),
\item there exists \(R > 0\) such that for every \(n \in \N\), \(\supp{\varphi_n} \subset B[0; R]\).
\end{enumerate}
The construction of the sequence \((\varphi_n)_{n \in \N}\) can be found for instance in  \cite{DePauw}*{Lemma~3.2};
the constant \(C > 0\) such that for every \(n \in \N\),
\[
\norm{D\varphi_n}_{L^1(\R^N)} \le C
\]
can be chosen to be of the order of \(\cH^{N-1}(K)\) if \(\cH^{N-1}(K) > 0\).
The construction in \cite{DePauw} gives \(\varphi_n \ge 1\) in \(K\); property \((d)\) is achieved by composition with a smooth Lipschitz function \(\Phi : \R \to \R \) such that \(\Phi(t) = 1\) for \(t \ge 1\).

\medskip

By property \((b)\) above, for every \(n \in \N\) we have
\[
\biggabs{\int\limits_{\R^N} \varphi_n \dif\nu}
\le \norm{\div{W}}_{L^\infty(B[0; R])} \norm{\varphi_n}_{L^1(\R^N)} + C \norm{W - V}_{L^\infty(B[0; R])}.
\]
As \(n\) tends to infinity, we get by the Dominated convergence theorem,
\[
\abs{\nu(K)}
\le C \norm{W - V}_{L^\infty(B[0; R])}.
\]
By the Weierstrass approximation theorem, for every \(\epsilon > 0\) we may choose \(W\) such that
\[
C \norm{W - V}_{L^\infty(B[0; R])} \le \epsilon.
\]
Thus,
\[
\abs{\nu(K)} \le \epsilon.
\]
Since \(\epsilon\) is arbitrary, we deduce that \(\nu(K) = 0\).
By inner regularity of the measure \(\nu\), property \eqref{eqConditionSigmaFinite} is satisfied by \(\nu\) for every Borel set, not necessarily compact.
\end{proof}

We deduce from Corollary~\ref{corollaryCharges} that the positive and negative parts of a measure which is a charge are also charges.
We conclude with the following example of a measure --- actually an \(L^1\) function --- which is a strong charge but whose positive and negative parts are not strong charges.

\begin{example}
\label{exampleStrongCharge}
Given \(\alpha > 0\), let \(u_\alpha : \R^N \to \R\) be the function defined for \(x \in \R\) by 
\[
u_\alpha(x) =
\begin{cases}
\abs{x} \sin{\frac{1}{\abs{x}^\alpha}}	& \text{if \(x \ne 0\),}\\
0	& \text{if \(x = 0\).}
\end{cases}
\]
Then, \(u_\alpha\) is continuous and belongs to \(W^{1, 1}\loc(\R^N)\) for \(\alpha < N\).
In particular, if \(W : \R^N \to \R^N\) is a smooth vector field with compact support, then the function 
\[
f_\alpha = \div{(u_\alpha W)}
\]
belongs to \(L^1(\R^N)\) and defines a strong charge since the vector field \(u_\alpha W\) is continuous.
However, if \(\alpha \ge 1\) and \(W(0) \ne 0\), then
\[
\lim_{r \to 0}{\frac{1}{r^{N-1}} \int\limits_{B(0; r)} f_\alpha^+} > 0.
\]
In particular, condition~\eqref{eqPhucTorresContinuous} is not satisfied, whence \(f_\alpha^+\) does not define a strong charge. 
\end{example}


\section*{Acknowledgments}

The author would like to thank Laurent Moonens for interesting discussions and for bringing the problem of removable singularities to his attention.
This work was supported by the Fonds de la Recherche scientifique---FNRS.




\begin{bibdiv}

\begin{biblist}

\bib{BLOP:05}{article}{
      author={Bartolucci, Daniele},
      author={Leoni, Fabiana},
      author={Orsina, Luigi},
      author={Ponce, Augusto~C.},
       title={Semilinear equations with exponential nonlinearity and measure
  data},
        date={2005},
     journal={Ann. Inst. H. Poincar\'e Anal. Non Lin\'eaire},
      volume={22},
       pages={799\ndash 815},
}


\bib{Besicovitch}{article}{
   author={Besicovitch, A. S.},
   title={On the definition of tangents to sets of infinite linear measure},
   journal={Proc. Cambridge Philos. Soc.},
   volume={52},
   date={1956},
   pages={20--29},
}


\bib{BocGalOrs:96}{article}{
      author={Boccardo, Lucio},
      author={Gallou{\"e}t, Thierry},
      author={Orsina, Luigi},
       title={Existence and uniqueness of entropy solutions for nonlinear
  elliptic equations with measure data},
        date={1996},
     journal={Ann. Inst. H. Poincar\'e Anal. Non Lin\'eaire},
      volume={13},
       pages={539\ndash 551},
}


\bib{Bourgain_Brezis}{article}{
   author={Bourgain, Jean},
   author={Brezis, Ha{\"{\i}}m},
   title={On the equation ${\rm div}\, Y=f$ and application to control of
   phases},
   journal={J. Amer. Math. Soc.},
   volume={16},
   date={2003},
   pages={393--426},
}


\bib{Brezis_Marcus_Ponce}{article}{
   author={Brezis, Haïm},
   author={Marcus, Marcus},
   author={Ponce, Augusto C.},
   title={Nonlinear elliptic equations with measures revisited},
   conference={
      title={Mathematical aspects of nonlinear dispersive equations},
   },
   book={
      editor={{B}ourgain, {J}.}, 
      editor={{K}enig, {C}.},
      editor={{K}lainerman, {S}.},     
      series={Ann. of Math. Stud.},
      volume={163},
      publisher={Princeton Univ. Press},
      place={Princeton, NJ},
   },
   date={2007},
   pages={55--109},
}


\bib{Carleson}{book}{
      author={Carleson, Lennart},
       title={Selected problems on exceptional sets},
      series={Van Nostrand Mathematical Studies, No. 13},
   publisher={Van Nostrand},
     address={Princeton},
        date={1967},
}


\bib{Del:1915}{article}{
      author={de~la Vall{\'e}e~Poussin, C.},
       title={Sur l'int\'egrale de {L}ebesgue},
        date={1915},
     journal={Trans. Amer. Math. Soc.},
      volume={16},
       pages={435\ndash 501},
}


\bib{DePauw}{article}{
   author={De Pauw, Thierry},
   title={On the exceptional sets of the flux of a bounded vector field},
   journal={J. Math. Pures Appl. (9)},
   volume={82},
   date={2003},
   pages={1191--1217},
}

\bib{DePauw_Pfeffer}{article}{
   author={De Pauw, Thierry},
   author={Pfeffer, Washek F.},
   title={Distributions for which ${\rm div}\,v=F$ has a continuous
   solution},
   journal={Comm. Pure Appl. Math.},
   volume={61},
   date={2008},
   pages={230--260},
}

\bib{DePauw_Torres}{article}{
   author={De Pauw, Thierry},
   author={Torres, Monica},
   title={On the distributional divergence of vector fields vanishing at
   infinity},
   journal={Proc. Roy. Soc. Edinburgh Sect. A},
   volume={141},
   date={2011},
   pages={65--76},
}


\bib{deValeriola_Moonens}{article}{
   author={de Valeriola, S{\'e}bastien},
   author={Moonens, Laurent},
   title={Removable sets for the flux of continuous vector fields},
   journal={Proc. Amer. Math. Soc.},
   volume={138},
   date={2010},
   pages={655--661},
}


\bib{Fol:99}{book}{
      author={Folland, Gerald~B.},
       title={Real analysis},
      series={Pure and Applied Mathematics},
   publisher={John Wiley \& Sons Inc.},
     address={New York},
        date={1999},
        ISBN={0-471-31716-0},
}


\bib{Frostman}{book}{
      author={Frostman, Otto},
       title={Potentiel d'{\'e}quilibre et capacit{\'e} des ensembles avec
  quelques applications {\`a} la th{\'e}orie des fonctions},
   publisher={Meddelanden Mat. Sem. Univ. Lund 3, 115 s},
        date={1935},
}

\bib{Hausdorff}{article}{
      author={Hausdorff, Felix},
       title={Dimension und \"au\ss eres {M}a\ss},
        date={1918},
     journal={Math. Ann.},
      volume={79},
       pages={157\ndash 179},
}

\bib{Mattila}{book}{
   author={Mattila, Pertti},
   title={Geometry of sets and measures in Euclidean spaces},
   series={Cambridge Studies in Advanced Mathematics},
   volume={44},
   publisher={Cambridge University Press},
   place={Cambridge},
   date={1995},
   pages={xii+343},
}

\bib{Moonens}{article}{
   author={Moonens, Laurent},
   title={Removable singularities for the equation ${\rm div}\,v=0$},
   journal={Real Anal. Exchange 30th Summer Symposium Conference},
   date={2006},
   pages={125--132},
}


\bib{Pfeffer}{book}{
  author={Pfeffer, Washek F.},
  title={The Divergence Theorem and Sets of Finite Perimeter},
  series={Pure and Applied Mathematics},
  publisher = {CRC Press},
  place = {Boca Raton, FL},
  date={2012},
}

\bib{Phuc_Torres}{article}{
   author={Phuc, Nguyen Cong},
   author={Torres, Monica},
   title={Characterizations of the existence and removable singularities of
   divergence-measure vector fields},
   journal={Indiana Univ. Math. J.},
   volume={57},
   date={2008},
   pages={1573--1597},
}

\bib{Ponce}{book}{
  author={Ponce, Augusto C.},
  title={Selected problems on elliptic equations involving measures},
  date={2012},
  note={Available at \texttt{http://arxiv.org/abs/1204.0668}},
}


\bib{Rogers_62}{article}{
   author={Rogers, C. A.},
   title={Sets non-$\sigma $-finite for Hausdorff measures},
   journal={Mathematika},
   volume={9},
   date={1962},
   pages={95--103},
}

\bib{Rogers_70}{book}{
   author={Rogers, C. A.},
   title={Hausdorff measures},
   publisher={Cambridge University Press},
   place={London},
   date={1970},
   pages={viii+179},
}

\end{biblist}

\end{bibdiv}
\end{document}